\documentclass{amsart}
\usepackage{textcomp}
\usepackage{indentfirst}
\usepackage{amsmath}
\usepackage{amsthm}
\usepackage{amsfonts}
\usepackage{amssymb}
\usepackage{mathrsfs}
\usepackage{stmaryrd}
\usepackage{bbm}
\usepackage[T1]{fontenc}
\usepackage{multirow}

\newtheorem{thm}{Theorem}[section]
\newtheorem{cor}[thm]{Corollary}
\newtheorem{lem}[thm]{Lemma}
\newtheorem{prop}[thm]{Proposition}
\theoremstyle{definition}

\newtheorem{rem}[thm]{Remark}
\numberwithin{equation}{section}

\begin{document}

\title{Asymmetric truncated Toeplitz operators on finite-dimensional spaces}
\author[Joanna Jurasik]{Joanna Jurasik}%\\
\address{Department of Mathematics, Maria Curie-Sk\l odowska University, Maria Curie-Sk\l odowska Square 1, 20-031 Lublin, Poland}%
\email{asia.blicharz@op.pl}%
\author[Bartosz \L anucha]{Bartosz \L anucha}%\\
\address{Department of Mathematics, Maria Curie-Sk\l odowska University, Maria Curie-Sk\l odowska Square 1, 20-031 Lublin, Poland}%
\email{bartosz.lanucha@poczta.umcs.lublin.pl}%

%Department of Mathematics, Maria Curie-Sklodowska University\\
%20-031 Lublin, Poland\\
%E-mail: bartosz.lanucha@gmail.com}

\begin{abstract}
In this paper we consider asymmetric truncated Toeplitz operators acting between two finite-dimensional model spaces. We compute the dimension of the space of all such operators. We also describe the matrix representations of asymmetric truncated
Toeplitz operators acting between two finite-dimensional model spaces. Our results are generalizations of the results known for truncated Toeplitz operators.
\end{abstract}

\maketitle

\renewcommand{\thefootnote}{}



\renewcommand{\thefootnote}{\arabic{footnote}}
\setcounter{footnote}{0}

\section{Introduction}

Let $H^2$ be the classical Hardy space of the unit disk
$\mathbb{D}=\{z\colon|z|<1\}$. The Hardy space $H^2$ consists of functions
$f(z)=\sum_{k=0}^{\infty}\hat{f}(k)z^k$ analytic in $\mathbb{D}$ and
such that
$$\|f\|^2=\sum_{k=0}^{\infty}|\hat{f}(k)|^2<\infty.$$

As usual, $H^2$ will be identified via boundary values with a closed
subspace of $L^2(\partial\mathbb{D})$. %, namely the subspace of
%functions whose Fourier coefficients with negative indices vanish.
The orthogonal projection from $L^2(\partial\mathbb{D})$ onto $H^2$
will be denoted by $P$. The Szeg\"{o} projection $P$ is an integral operator given by
$$Pf(z)=\frac{1}{2\pi}\int_0^{2\pi}\frac{f(e^{it})dt}{1-e^{-it}z},\quad z\in \mathbb{D}.$$
Note that the above integral defines a function $Pf$ analytic in $\mathbb{D}$ for every $f\in L^1(\partial\mathbb{D})$.

The classical Toeplitz operator $T_{\varphi}$ with symbol
$\varphi\in L^2(\partial\mathbb{D})$ is defined on $H^2$ by
$$T_{\varphi}f=P(\varphi f).$$
The operator $T_{\varphi}$ is bounded if and only if $\varphi\in
L^{\infty}(\partial\mathbb{D})$. An important example of a classical Toeplitz
operator is the unilateral shift $S=T_z$, that is, $Sf(z)=zf(z)$. Its Hilbert space adjoint $S^{*}$,
called the backward shift, is given by
$$S^{*}f(z)=\frac{f(z)-f(0)}{z}.$$

Let $H^{\infty}$ be the algebra of bounded analytic functions on
$\mathbb{D}$ and let $\alpha$ be an arbitrary inner function, that
is, $\alpha\in H^{\infty}$ and $|\alpha|=1$ a.e. on $\partial
\mathbb{D}$. The model space corresponding to $\alpha$ is the closed
subspace $K_{\alpha}$ of $H^2$ of the form
$$K_{\alpha}=H^2\ominus \alpha H^2.$$
The theorem of A. Beurling (see for instance \cite[Thm.
8.1.1]{r}) implies that every nontrivial $S^{*}$-invariant subspace of $H^2$ is a model space $K_{\alpha}$ corresponding to some inner function $\alpha$.

The model space $K_{\alpha}$ is a reproducing kernel Hilbert space
with the kernel function given by
\begin{equation}\label{kerker}
k_{w}^{\alpha}(z)=\frac{1-\overline{\alpha(w)}\alpha(z)}{1-\overline{w}z},\quad w, z\in\mathbb{D}.
\end{equation}
In other words, $f(w)=\langle f, k_{w}^{\alpha}\rangle$ for every
$f\in K_{\alpha}$ and $w\in\mathbb{D}$ (here $\langle
\cdot,\cdot\rangle$ is the usual integral inner product). Note that
$k_{w}^{\alpha}$ is bounded, and so the subspace
$K_{\alpha}^{\infty}$ of all bounded functions in $K_{\alpha}$ is
dense in $K_{\alpha}$. If $\alpha(w)=0$, then $k_{w}^{\alpha}$ is equal to the Szeg\"{o} kernel $k_w(z)=(1-\overline{w}z)^{-1}$.

We say that $\alpha$ has an angular derivative in the sense of Carath\'{e}odory (an ADC) at the point $\eta\in\partial\mathbb{D}$ if $\alpha$ has a nontangential limit $\alpha(\eta)$ at $\eta$ ($\alpha(z)$ tends to $\alpha(\eta)$ as $z$ tends to $\eta$ with $|z-\eta|/(1-|z|)$ bounded), $|\alpha(\eta)|=1$ and the difference quotient $(\alpha(z)-\alpha(\eta))/(z-\eta)$ has a nontangential limit at $\eta$ (for more details see \cite[pp. 33--37]{bros}). It was proved by P.R. Ahern and D.N. Clark in \cite{clark,cclark} that $\alpha$ has an ADC at some point $\eta\in\partial\mathbb{D}$ if and only if every $f\in K_{\alpha}$ has a nontangential limit $f(\eta)$ at $\eta$. If $\alpha$ has an ADC at $\eta$, then the function $k_{\eta}^{\alpha}$ defined by \eqref{kerker} with $\eta$ in place of $w$, belongs to $K_{\alpha}$. Moreover, $f(\eta)=\langle f, k_{\eta}^{\alpha}\rangle$ for every $f\in K_{\alpha}$ and $k_{w}^{\alpha}\rightarrow k_{\eta}^{\alpha}$ (in norm) as $w$ tends to $\eta$ nontangentially. Note that
\begin{equation*}
\|k_{\eta}^{\alpha}\|^2=k_{\eta}^{\alpha}(\eta)=|\alpha'(\eta)|.
\end{equation*}

The orthogonal projection
$P_{\alpha}$ from $L^2(\partial\mathbb{D})$ onto $K_{\alpha}$ is
given by $$P_{\alpha}f(z)=\langle f,k_{z}^{\alpha}\rangle,\quad  z\in\mathbb{D}.$$ Like the Szeg\"{o} projection $P$, $P_{\alpha}$ is an integral operator and $P_{\alpha}f$ is a well defined analytic function for all $f\in L^1(\partial\mathbb{D})$.

If $\alpha$ is an inner function, then the formula
\begin{equation}\label{numerek3}
C_{\alpha}f(z)=\alpha(z)\overline{z}\overline{f(z)},\quad |z|=1,
\end{equation}
defines a conjugation on $L^2(\partial\mathbb{D})$. In other words, $C_{\alpha}\colon L^2(\partial\mathbb{D})\rightarrow L^2(\partial\mathbb{D})$ is an antilinear, isometric involution
(see \cite[Subection 2.3]{s}). Actually, it can be
verified that $C_{\alpha}$ maps $\alpha H^2$
onto $\overline{H_0^2}$, $\overline{H_0^2}$ onto $\alpha H^2$, and
preserves $K_{\alpha}$. For $f\in K_{\alpha}$ we will write $\widetilde{f}$ in place of
$C_{\alpha}f$ when no confusion can arise. A simple calculation reveals that the conjugate kernel
$\widetilde{k}_{w}^{\alpha}$ is given by
$$\widetilde{k}_{w}^{\alpha}(z)=\frac{\alpha(z)-\alpha(w)}{z-w},\quad w,z\in\mathbb{D}.$$
Moreover, if $\eta\in\partial\mathbb{D}$ and $k_{\eta}^{\alpha}\in K_{\alpha}$, then
\begin{equation}\label{kerk}
\widetilde{k}_{\eta}^{\alpha}(z)=\frac{\alpha(z)-\alpha(\eta)}{z-\eta}=\alpha(\eta)\overline{\eta}k_{\eta}^{\alpha}(z).
\end{equation}

 A truncated Toeplitz operator $A_{\varphi}^{\alpha}$ with a symbol $\varphi\in
L^2(\partial\mathbb{D})$ is the compression of $T_{\varphi}$ to the model space $K_{\alpha}$. More precisely, $A_{\varphi}^{\alpha}$ is defined on $K_{\alpha}$ by
$$A_{\varphi}^{\alpha}f=P_{\alpha}(\varphi f).$$

An extensive study of truncated Toeplitz operators began in 2007
with D. Sarason's paper \cite{s}. Despite similar definitions, truncated Toeplitz
operators differ form the classical ones in many ways. For example, $T_{\varphi}=0$ if and only if $\varphi=0$, but $A_{\varphi}^{\alpha}=0$ if and only if $\varphi\in \overline{\alpha H^2}+\alpha H^2$ (see \cite[Thm. 3.1]{s}). Moreover, unlike in the classical case, unbounded symbols can produce bounded truncated Toeplitz operators
and there are bounded truncated Toeplitz operators for which no
bounded symbol exists (see \cite{bar} for more details). More
interesting results about truncated Toeplitz operators can be found in \cite{si,bros, gar,gar2,gar3}.

Recently, the authors in \cite{ptak} and \cite{part} introduced the so-called asymmetric truncated
Toeplitz operators, which are generalizations of truncated Toeplitz operators. Let $\alpha$, $\beta$ be two inner functions and let $\varphi\in L^2(\partial\mathbb{D})$. An asymmetric
truncated Toeplitz operator $A_{\varphi}^{\alpha,\beta}$ with a symbol $\varphi\in
L^2(\partial\mathbb{D})$ is the
operator from $K_{\alpha}$ into $K_{\beta}$ defined by
$$A_{\varphi}^{\alpha,\beta}f=P_{\beta}(\varphi f),\quad f\in K_{\alpha}.$$

The asymmetric truncated Toeplitz operator
$A_{\varphi}^{\alpha,\beta}$ is closed and densely defined. Obviously, $A_{\varphi}^{\alpha,\alpha}=A_{\varphi}^{\alpha}$.

Let
$$\mathcal{T}(\alpha,\beta)=\{A_{\varphi}^{\alpha,\beta}\ \colon\ \varphi\in
L^2(\partial\mathbb{D})\ \mathrm{and}\ A_{\varphi}^{\alpha,\beta}\
\mathrm{is\ bounded}\}$$
and $\mathcal{T}(\alpha)=\mathcal{T}(\alpha,\alpha)$.

In 2008 \cite{w} J.A. Cima, W.T. Ross and W.R. Wogen considered truncated Toeplitz operators on finite-dimensional model spaces. It is known that $K_{\alpha}$ has finite dimension $m$ if and only if $\alpha$ is a finite Blaschke product of degree $m$. In that case every $f\in K_{\alpha}$ is analytic in a domain containing the closed unit disk  (see \cite[Prop. 5.7.6]{bros}). If $\alpha$ has $m$ distinct zeros $a_1,\ldots,a_m$, then the sets $\{k_{a_1}^{\alpha},\ldots,k_{a_m}^{\alpha}\}$ and $\{\widetilde{k}_{a_1}^{\alpha},\ldots,\widetilde{k}_{a_m}^{\alpha}\}$ are two (non-orthonormal) bases for $K_{\alpha}$. The authors in \cite{w} characterized the operators from $\mathcal{T}(\alpha)$ in terms of the matrix representations with respect to each of these bases. They showed that a matrix representing a truncated Toeplitz operator on $m$-dimensional model space is completely determined by $2m-1$ of its entries, those along the main diagonal and the first row (and the first row can be replaced by any other row or column). They also proved a similar result for the so-called Clark bases.

The main purpose of this paper is to generalize the results from \cite{w} to the case of asymmetric truncated Toeplitz operators.

In Section 2 we compute the dimension of $\mathcal{T}(\alpha,\beta)$ for two finite Blaschke products $\alpha$, $\beta$. D. Sarason \cite[Thm. 3.1]{s} proved that if $\alpha$ is a finite Blaschke product of degree $m>0$, then the dimension of $\mathcal{T}(\alpha)$ is $2m-1$. We show that if $\alpha$ and $\beta$ are finite Blachke products of degree $m>0$ and $n>0$, respectively, then the dimension of $\mathcal{T}(\alpha,\beta)$ is $m+n-1$.

In Section 3 we generalize the results from \cite{w} concerning matrix representations. We characterize matrix representations of asymmetric truncated Toeplitz operators acting between finite-dimensional model spaces. We consider matrix representations  with respect to kernel bases, conjugate kernel bases, Clark bases and modified Clark bases. In each of these cases we show how the matrix representing an asymmetric truncated Toeplitz operator is completely determined by $m+n-1$ of its entries.

\section{The dimension of $\mathscr{T}(\alpha,\beta)$}

Here we compute the dimension of the space of all asymmetric truncated Toeplitz operators acting between finite-dimensional model spaces. We also give examples of bases for $\mathscr{T}(\alpha,\beta)$ in this case. The proofs given here are analogous to those from \cite[Thm. 7.1]{s}.

As mentioned in the
Introduction, if $K_{\alpha}$ has finite dimension $m>0$, then the dimension of $\mathcal{T}(\alpha)$ is $2m-1$ (\cite[Thm. 7.1(a)]{s}). Here we prove the following.

\begin{prop}\label{dym}
Let $K_{\alpha}$ have finite dimension $m>0$ and let $K_{\beta}$ have finite dimension $n>0$. The dimension of
$\mathscr{T}(\alpha,\beta)$ is $m+n-1$.
\end{prop}

In the proof of Proposition \ref{dym} we use the fact that if $\alpha$, $\beta$ are two nonconstant inner functions, then $A_{\varphi}^{\alpha, \beta}=0$ if and only if $\varphi\in \overline{\alpha H^2}+\beta H^2$ (see \cite[Thm 2.1]{blicharz1} for proof). We also use the following simple lemma from \cite{blicharz1}.

\begin{lem}[\cite{blicharz1}, Lem. 2.2]\label{le1}
Let $\alpha$, $\beta$ be two arbitrary inner functions. If
\begin{equation*}%\label{9}
K_{\alpha}\subset \beta H^2,
\end{equation*}
then both $\alpha$ and $\beta$ have no zeros in $\mathbb{D}$, or at least one of the functions $\alpha$ or $\beta$ is a constant function.
\end{lem}

\begin{proof}[Proof of Proposition \ref{dym}]
By \cite[Cor. 2.6]{blicharz1}, every operator $A$ in $\mathscr{T}(\alpha,\beta)$ can be written as a sum $A=A_{\overline{\chi}}+A_{\psi}$ with $\chi\in K_{\alpha}$ and $\psi\in K_{\beta}$. Since $K_{\alpha}$ and $K_{\beta}$ have finite dimension, it follows that $\alpha$ and $\beta$ are finite Blaschke products and $K_{\alpha}\subset H^{\infty}$, $K_{\beta}\subset H^{\infty}$. Consequently, $\mathscr{T}(\alpha,\beta)$ is spanned by its subspaces
\begin{displaymath}
\mathscr{T}_{\infty}(\alpha,\beta)=\{A_{\varphi}^{\alpha, \beta}: \varphi \in {H}^{\infty}\}
\end{displaymath}
and
\begin{displaymath}
\mathscr{T}_{\overline{\infty}}(\alpha,\beta)=\{A_{\varphi}^{\alpha, \beta}: \varphi \in \overline{H^{\infty}}\}.
\end{displaymath}

We first compute the dimension of $\mathscr{T}_{\infty}(\alpha,\beta)$ and the dimension of $\mathscr{T}_{\overline{\infty}}(\alpha,\beta)$. To this end, we consider the linear mapping
$\varphi\mapsto A_{\varphi}^{\alpha, \beta}$ acting from ${H}^{\infty}$ onto $\mathscr{T}_{\infty}(\alpha,\beta)$. By \cite[Thm 2.1]{blicharz1}, its kernel is equal to $\beta H^{\infty}$. Indeed, if $\varphi\in\beta H^{\infty}$, then $A_{\varphi}^{\alpha, \beta}=0$. On the other hand, if $\varphi\in H^{\infty}$ and $A_{\varphi}^{\alpha, \beta}=0$, then $\varphi=\overline{\alpha h}_1+\beta h_2$ for some $h_1, h_2\in H^2$. Hence $\varphi-\beta h_2=\overline{\alpha h}_1$ is a constant function, $\varphi=\beta h_2+c$ for some complex number $c$, and
$$0=A_{\varphi}^{\alpha, \beta}=A_{\beta h_2+c}^{\alpha, \beta}=cP_{\beta|K_{\alpha}}.$$
If $c\neq 0$, then the above implies that $K_{\alpha}\subset\beta H^2$, which, by Lemma \ref{le1}, never happens for nonconstant Blaschke products $\alpha$, $\beta$. Therefore $c=0$ and $\varphi\in \beta H^{\infty}$. From this
\begin{displaymath}
\dim \mathscr{T}_{\infty}(\alpha,\beta)=\dim\left({H}^{\infty}/{\beta H^{\infty}}\right)=n.
\end{displaymath}
Similarly, the mapping $\varphi\mapsto A_{\varphi}^{\alpha, \beta}$ acting from $\overline{H^{\infty}} $ onto $\mathscr{T}_{\overline{\infty}}(\alpha,\beta)$ has kernel equal to $\overline{\alpha H^{\infty}}$ and
\begin{displaymath}
\dim \mathscr{T}_{\overline{\infty}}(\alpha,\beta)=\dim \left(\overline{H^{\infty}}/\overline{\alpha H^{\infty}}\right)=m.
\end{displaymath}

To complete the proof, we only need to show that
\begin{equation}\label{czy}
\dim (\mathscr{T}_{\infty}(\alpha,\beta)\cap \mathscr{T}_{\overline{\infty}}(\alpha,\beta))=1,
\end{equation}
for then
\begin{displaymath}
\begin{split}
\dim \mathscr{T}(\alpha,\beta)&=\dim \mathscr{T}_{\infty}(\alpha,\beta)+\dim \mathscr{T}_{\overline{\infty}}(\alpha,\beta)-\dim (\mathscr{T}_{\infty}(\alpha,\beta)\cap \mathscr{T}_{\overline{\infty}}(\alpha,\beta))\\
&=m+n-1.
\end{split}
\end{displaymath}

Note that here $A_{1}^{\alpha, \beta}\neq 0$. Otherwise, we would have $K_{\alpha}\subset\beta H^2$, which, by Lemma \ref{le1} again, is impossible for nonconstant Blaschke products $\alpha$, $\beta$. Clearly, $A_{1}^{\alpha, \beta}\in \mathscr{T}_{\infty}(\alpha,\beta)\cap \mathscr{T}_{\overline{\infty}}(\alpha,\beta)$, so
$\mathscr{T}_{\infty}(\alpha,\beta)\cap \mathscr{T}_{\overline{\infty}}(\alpha,\beta)\neq\{0\}$.

Assume now that $A\in \mathscr{T}_{\infty}(\alpha,\beta)\cap \mathscr{T}_{\overline{\infty}}(\alpha,\beta)$, say $A=A_{\varphi_1}^{\alpha, \beta}=A_{\overline{\varphi}_2}^{\alpha, \beta}$, $\varphi_1,\varphi_2\in {H}^{\infty}$. Then
\begin{displaymath}
A_{\varphi_1-\overline{\varphi}_2}^{\alpha, \beta}=0,
\end{displaymath}
and $\varphi_1-\overline{\varphi}_2\in \overline{\alpha H^2}+\beta H^2$ by \cite[Thm 2.1]{blicharz1}. In other words, there exist $h_1,h_2\in H^2$ such that
\begin{displaymath}
\varphi_1-\overline{\varphi}_2=\overline{\alpha h}_1+\beta h_2
\end{displaymath} or
\begin{displaymath}
\varphi_1-\beta h_2=\overline{\varphi_2+\alpha h_1}.
\end{displaymath}
This implies that there exists a complex number $c$ such that $\varphi_1=c+\beta h_2$ and $A=A_{c+\beta h_2}^{\alpha, \beta}=cA_{1}^{\alpha, \beta}$. Thus every operator in $\mathscr{T}_{\infty}(\alpha,\beta)\cap \mathscr{T}_{\overline{\infty}}(\alpha,\beta)$ is a scalar multiple of $A_{1}^{\alpha, \beta}$ and \eqref{czy} holds.
\end{proof}

Recall the following.

\begin{prop}[\cite{blicharz1}, Prop. 3.1]\label{rankone}
Let $ \alpha$, $\beta$ be two nonconstant inner functions.
\begin{itemize}
\item[(a)] For $w\in \mathbb{D}$, the operators $\widetilde{k}_w^{\beta}\otimes {k}_w^{\alpha}$ and
$k_w^{\beta}\otimes \widetilde{k}_w^{\alpha}$ belong to $\mathscr{T}(\alpha,\beta)$, $$\widetilde{k}_w^{\beta}\otimes {k}_w^{\alpha}=A_{\frac{{\beta(z)}}{{z}-{w}}}^{\alpha, \beta}\quad \mathrm{and}\quad k_w^{\beta}\otimes \widetilde{k}_w^{\alpha}=A_{\frac{\overline{\alpha(z)}}{\overline{z}-\overline{w}}}^{\alpha, \beta}.$$
\item[(b)] If both $\alpha$ and $\beta$ have an ADC at the point $\eta$ of $\partial\mathbb{D}$, then the operator $k_{\eta}^{\beta}\otimes {k}_{\eta}^{\alpha} $ belongs to $\mathscr{T}(\alpha,\beta)$, $$k_{\eta}^{\beta}\otimes {k}_{\eta}^{\alpha} =A_{k_{\eta}^{\beta}+\overline{{k}}_{\eta}^{\alpha}-1}^{\alpha, \beta}.$$
\end{itemize}
\end{prop}

As mentioned before, a finite Blaschke product is analytic in a domain containing the closed unit disk $\overline{\mathbb{D}}=\{z\colon|z|\leq1\}$, and therefore has an ADC at every $\eta\in\partial\mathbb{D}$. Consequently, if $\alpha$ and $\beta$ are two finite Blaschke products, then $k_{\eta}^{\alpha}\in K_{\alpha}$, $k_{\eta}^{\beta}\in K_{\beta}$ and $k_{\eta}^{\beta}\otimes {k}_{\eta}^{\alpha} $ belongs to $\mathscr{T}(\alpha,\beta)$ for all $\eta\in\partial\mathbb{D}$. Moreover, by \eqref{kerk},
\begin{displaymath}
k_{\eta}^{\beta}\otimes \widetilde{k}_{\eta}^{\alpha}=\overline{\alpha(\eta)}\eta k_{\eta}^{\beta}\otimes k_{\eta}^{\alpha}\quad\mathrm{and}\quad \widetilde{k}_{\eta}^{\beta}\otimes k_{\eta}^{\alpha}=\alpha(\eta)\overline{\eta} k_{\eta}^{\beta}\otimes k_{\eta}^{\alpha}.
\end{displaymath}

\begin{cor}\label{bazy}
Let $K_{\alpha}$ have finite dimension $m>0$ and let $K_{\beta}$ have finite dimension $n>0$. If $w_1,\ldots,w_{m+n-1}$ are distinct points in the closed unit disk $\overline{\mathbb{D}}$, then:
\begin{itemize}
\item[(a)] the operators $\widetilde{k}_{w_j}^{\beta}\otimes k_{w_j}^{\alpha}$, $j=1,\ldots,m+n-1$, form a basis for $\mathcal{T}(\alpha,\beta)$;
\item[(b)] the operators $k_{w_j}^{\beta}\otimes \widetilde{k}_{w_j}^{\alpha}$, $j=1,\ldots,m+n-1$, form a basis for $\mathcal{T}(\alpha,\beta)$.
\end{itemize}
\end{cor}
\begin{proof}
We only prove part $(\mathrm{a})$ of the theorem. Proof of part $(\mathrm{b})$ is similar (compare with \cite[Thm. 7.1(b)]{s} and \cite[Lem. 3.1]{w}).

Let $w_1,\ldots,w_{m+n-1}$ be distinct points in $\overline{\mathbb{D}}$. By Proposition \ref{rankone}, the operators $\widetilde{k}_{w_j}^{\beta}\otimes k_{w_j}^{\alpha}$, $j=1,\ldots,m+n-1$, belong to $\mathcal{T}(\alpha,\beta)$. Since the dimension of $\mathcal{T}(\alpha,\beta)$ is $m+n-1$, it is enough to prove that these operators are linearly independent.

Assume that
$$\sum_{j=1}^{m+n-1}c_j\widetilde{k}_{w_j}^{\beta}\otimes k_{w_j}^{\alpha}=0$$
for some scalars $c_1,\ldots,c_{m+n-1}$. We first show that $c_1=0$.

Since the functions $k_{w_1}^{\alpha},\ldots, k_{w_m}^{\alpha}$ are linearly independent (see \cite[p. 509]{s}), there exists $f\in K_{\alpha}$ such that
$$\langle f,k_{w_j}^{\alpha}\rangle=\left\{\begin{array}{cc} 1&\mathrm{for}\ j= 1,\\0&\mathrm{for}\ 1<j\leq m,\end{array}\right.$$
and
\begin{displaymath}
0=\sum_{j=1}^{m+n-1}c_j\widetilde{k}_{w_j}^{\beta}\otimes k_{w_j}^{\alpha}(f)=c_1\widetilde{k}_{w_1}^{\beta}+\sum_{j=m+1}^{m+n-1}c_jf(w_j)\widetilde{k}_{w_j}^{\beta}.
\end{displaymath}
But $\widetilde{k}_{w_1}^{\beta}, \widetilde{k}_{w_{m+1}}^{\beta},\ldots, \widetilde{k}_{w_{m+n-1}}^{\beta}$ also are linearly independent, so $c_1=0$.

A similar reasoning shows that $c_j=0$ for every $j=1,\ldots,m+n-1$, which completes the proof.
\end{proof}

\section{Matrix representations}

For the reminder of the paper we assume that $\alpha$ and $\beta$ are two finite Blaschke products with zeros $a_1,\ldots, a_m$ and $b_1,\ldots, b_n$, respectively, that is,
\begin{equation}\label{blaszki}
\alpha(z)=\prod_{i=1}^{m}\frac{a_i-z}{1-\overline{a}_iz},\qquad \beta(z)=\prod_{j=1}^{n}\frac{b_j-z}{1-\overline{b}_jz}.
\end{equation}

\subsection{Kernel bases and conjugate kernel bases}

Let $\alpha$ and $\beta$ be given by \eqref{blaszki}. Here we assume that the zeros $a_1,\ldots, a_m$ are distinct and that the zeros $b_1,\ldots, b_n$ are distinct. Then the kernel functions $\{k_{a_1}^{\alpha},\ldots,k_{a_m}^{\alpha}\}$ form a basis for $K_{\alpha}$ and so do the conjugate kernel functions $\{\widetilde{k}_{a_1}^{\alpha},\ldots,\widetilde{k}_{a_m}^{\alpha}\}$. Similarly, $\{k_{b_1}^{\beta},\ldots,k_{b_n}^{\beta}\}$ and $\{\widetilde{k}_{b_1}^{\beta},\ldots,\widetilde{k}_{b_n}^{\beta}\}$ are bases for $K_{\beta}$.

Of course, it is possible that $\alpha$ and $\beta$ have some zeros in common. In this subsection we assume that $\alpha$ and $\beta$ have precisely $l$ zeros in common ($l=0$ if there are no zeros in common), those zeros being $a_i=b_i$ for $i\leq l$.

Let $A$ be any linear operator from $K_{\alpha}$ into $K_{\beta}$. Its matrix representation $M_A=(r_{s,p})$ with respect to the kernel bases $\{k_{a_1}^{\alpha},\ldots,k_{a_m}^{\alpha}\}$ and $\{k_{b_1}^{\beta},\ldots,k_{b_n}^{\beta}\}$ is determined by
$$Ak_{a_p}^{\alpha}=\sum_{j=1}^{n}r_{j,p}k_{b_j}^{\beta},\quad p=1,\ldots,m.$$
Since
$$\langle k_{b_j}^{\beta},\widetilde{k}_{b_s}^{\beta}\rangle=\left\{\begin{array}{cc} \overline{\beta'(b_s)}&\mathrm{for}\ j= s,\\0&\mathrm{for}\ j\neq s,\end{array}\right.$$
we have
\begin{equation*}
r_{s,p}=(\overline{\beta'(b_s)})^{-1}\langle A k^{\alpha}_{a_p}, \widetilde{k}_{b_s}^{\beta}\rangle .
\end{equation*}

Similarly, if $\widetilde{M}_A=(t_{s,p})$ is the matrix representation of $A$ with respect to the conjugate kernel bases $\{\widetilde{k}_{a_1}^{\alpha},\ldots,\widetilde{k}_{a_m}^{\alpha}\}$ and $\{\widetilde{k}_{b_1}^{\beta},\ldots,\widetilde{k}_{b_n}^{\beta}\}$, then
\begin{equation*}
t_{s,p}= \beta'(b_s)^{-1}\langle A \widetilde{k}_{a_p}^{\alpha}, k_{b_s}^{\beta}\rangle .
\end{equation*}

\begin{thm}\label{macierz}
Let the function $\alpha$ be a finite Blaschke product with $m$ distinct zeros $a_1,\ldots, a_m$, let $\beta$ be a finite Blaschke product with $n$ distinct zeros $b_1,\ldots, b_n$ and assume that $\alpha$ and $\beta$ have precisely $l$ zeros in common: $a_i=b_i$ for $i\leq l$ ($l=0$ if there are no zeros in common). Let $A$ be any linear transformation from $K_{\alpha}$ into $K_{\beta}$. If $M_A=(r_{s,p})$ is the matrix representation of $A$ with respect to the bases $\{k_{a_1}^{\alpha},\ldots, k_{a_m}^{\alpha} \}$ and $\{k_{b_1}^{\beta},\ldots, k_{b_n}^{\beta} \}$, and
  \vspace{0.2cm}
\begin{itemize}
\item[(a)] $l=0 $, then $A\in \mathscr{T}(\alpha,\beta)$ if and only if
\begin{equation}\label{1}
r_{s,p}=\frac{\overline{\beta'(b_s)}(\overline{a}_1-\overline{b}_s)r_{s,1}+\overline{\beta'(b_1)}(\overline{b}_1-\overline{a}_1)r_{1,1}+\overline{\beta'(b_1)}(\overline{a}_p-\overline{b}_1)r_{1,p}}{\overline{\beta'(b_s)}(\overline{a}_p-\overline{b}_s)}
\end{equation} for all $1\leq  p\leq m$ and $1\leq s \leq n$;
  \vspace{0.3cm}

\item[(b)] $l>0$, then $A\in \mathscr{T}(\alpha,\beta)$ if and only if
\begin{equation}\label{2}
r_{s,p}=\frac{\overline{\beta'(b_1)}(\overline{a}_1-\overline{b}_s)r_{1,s}+\overline{\beta'(b_1)}(\overline{a}_p-\overline{b}_1)r_{1,p}}{\overline{\beta'(b_s)}(\overline{a}_p-\overline{b}_s)}
\end{equation} for all $p,s$ such that $1\leq  p\leq m$, $1\leq s \leq l$, $s\neq p$, and
\begin{equation}\label{3}
r_{s,p}=\frac{\overline{\beta'(b_s)}(\overline{a}_1-\overline{b}_s)r_{s,1}+\overline{\beta'(b_1)}(\overline{a}_p-\overline{b}_1)r_{1,p}}{\overline{\beta'(b_s)}(\overline{a}_p-\overline{b}_s)}
\end{equation} for all $p,s$ such that $1\leq p \leq m$, $l<  s\leq n$.
\end{itemize}
\end{thm}

\begin{proof}
We first prove the necessity of the conditions given in the theorem. Let $A=A_{\varphi}^{\alpha, \beta}$ be an asymmetric truncated Toeplitz operator with symbol $\varphi\in L^2(\partial\mathbb{D})$. Recall that the matrix representation $M_{A_{\varphi}^{\alpha, \beta}}=(r_{s,p})$ of $A_{\varphi}^{\alpha, \beta}$ with respect to the bases $\{k_{a_1}^{\alpha},\ldots, k_{a_m}^{\alpha} \}$ and $\{k_{b_1}^{\beta},\ldots, k_{b_n}^{\beta} \}$ is given by
\begin{displaymath}
r_{s,p}= (\overline{\beta'(b_s)})^{-1}\langle A_{\varphi}^{\alpha, \beta}k_{a_p}^{\alpha}, \widetilde{k}_{b_s}^{\beta}\rangle .
\end{displaymath}

By \cite[Cor. 2.6]{blicharz1},  $A_{\varphi}^{\alpha, \beta}$ can be written as
\begin{displaymath}
A_{\varphi}^{\alpha, \beta}=A_{\overline{\chi}+\psi}^{\alpha, \beta},
\end{displaymath}
for some $\chi \in K_{\alpha}$, $\psi \in K_{\beta}$. Since the functions $\widetilde{k}_{a_i}^{\alpha}$, $i=1,\ldots , m$, form a basis for $ K_{\alpha}$ and the functions $\widetilde{k}_{b_j}^{\beta}$, $j=1,\ldots , n$, form a basis for $ K_{\beta}$, we can write
\begin{displaymath}
\chi=\sum_{i=1}^{m}c_i\widetilde{k}_{a_i}^{\alpha},\quad \psi=\sum_{j=1}^{n}d_j\widetilde{k}_{b_j}^{\beta}.
\end{displaymath}
We now compute $r_{s,p}$ in terms of the scalars $c_1,\ldots,c_m$ and $d_1,\ldots,d_n$.

Since $\alpha(a_i)=0$ and $\beta(b_j)=0$, we have
\begin{displaymath}
\widetilde{k}_{a_i}^{\alpha}(z)=\frac{\alpha(z)}{z-a_i},\qquad  \widetilde{k}_{b_j}^{\beta}(z)=\frac{\beta(z)}{z-b_j}
\end{displaymath}
and
\begin{displaymath}
A_{\varphi}^{\alpha, \beta}=\sum_{i=1}^{m}\overline{c_i}A_{\frac{\overline{\alpha(z)}}{\overline{z}-\overline{a}_i}}^{\alpha, \beta}+\sum_{j=1}^{n}d_jA_{\frac{\beta(z)}{z-b_j}}^{\alpha, \beta}.
\end{displaymath}
By Proposition \ref{rankone}$(\mathrm{a})$,
\begin{displaymath}
A_{\frac{\overline{\alpha(z)}}{\overline{z}-\overline{a}_i}}^{\alpha, \beta}={k}_{a_i}^{\beta}\otimes\widetilde{k}_{a_i}^{\alpha}\quad  \mathrm{and}\quad A_{\frac{\beta(z)}{z-b_j}}^{\alpha, \beta}=\widetilde{k}_{b_j}^{\beta}\otimes{k}_{b_j}^{\alpha},
\end{displaymath}
and so
\begin{displaymath}
\begin{split}
A_{\varphi}^{\alpha, \beta}{k}_{a_p}^{\alpha}&=\sum_{i=1}^{m}\overline{c}_i\langle {k}_{a_p}^{\alpha},\widetilde{k}_{a_i}^{\alpha}\rangle {k}_{a_i}^{\beta}+\sum_{j=1}^{n}d_j\langle {k}_{a_p}^{\alpha},{k}_{b_j}^{\alpha}\rangle \widetilde{k}_{b_j}^{\beta}\\ &=\overline{c}_p\overline{\alpha'(a_p)}{k}_{a_p}^{\beta}+\sum_{j=1}^{n}\frac{d_j}{1-\overline{a}_pb_j}\widetilde{k}^{\beta}_{b_j}.
\end{split}
\end{displaymath}
The last equality follows from the fact that
$$\langle k_{a_p}^{\alpha},\widetilde{k}_{a_i}^{\alpha}\rangle=\left\{\begin{array}{cc} \overline{\alpha'(a_p)}&\mathrm{for}\ i= p,\\0&\mathrm{for}\ i\neq p.\end{array}\right.$$
Consequently,
\begin{displaymath}
\begin{split}
r_{s,p}&= (\overline{\beta'(b_s)})^{-1}\langle A_{\varphi}^{\alpha, \beta}k_{a_p}^{\alpha}, \widetilde{k}_{b_s}^{\beta}\rangle \\ &=\overline{c}_p\frac{\overline{\alpha'(a_p)}}{\overline{\beta'(b_s)}}\langle k_{a_p}^{\beta}, \widetilde{k}_{b_s}^{\beta}\rangle + \frac{1}{\overline{\beta'(b_s)}}\sum_{j=1}^{n}\frac{d_j}{1-\overline{a}_pb_j}\langle \widetilde{k}_{b_j}^{\beta},\widetilde{k}_{b_s}^{\beta}\rangle\\ &= \overline{c}_p\frac{\overline{\alpha'(a_p)}}{\overline{\beta'(b_s)}}\langle k_{a_p}^{\beta}, \widetilde{k}_{b_s}^{\beta}\rangle + \frac{1}{\overline{\beta'(b_s)}}\sum_{j=1}^{n}\frac{d_j}{(1-\overline{a}_pb_j)(1-\overline{b}_sb_j)}.
\end{split}
\end{displaymath}

$(\mathrm{a})$ $l=0$.

\vspace{0.1cm}

In this case $a_p\neq b_s$ for all $1\leq p\leq m$ and $1\leq s\leq n$. Therefore
$$
\langle k_{a_p}^{\beta}, \widetilde{k}_{b_s}^{\beta}\rangle=\frac{\overline{\beta(a_p)}}{\overline{a}_p-\overline{b}_s}\neq0
$$
and
\begin{displaymath}
r_{s,p}=\frac{\overline{c}_p}{\overline{a}_p-\overline{b}_s}\frac{\overline{\alpha'(a_p)}}{\overline{\beta'(b_s)}}\overline{\beta(a_p)}+ \frac{1}{\overline{\beta'(b_s)}}\sum_{j=1}^{n}\frac{d_j}{(1-\overline{a}_pb_j)(1-\overline{b}_sb_j)}
\end{displaymath}
for all $1\leq p\leq m$, $1\leq s\leq n$.

We now show that $r_{s,p}$ satisfies \eqref{1} for all $1\leq s\leq n$ and $1\leq p\leq m$. Clearly, \eqref{1} holds for $s=1$. Assume that $s\neq 1$.
Using the equality
 \begin{displaymath}
\frac{\overline{a}_p-\overline{b}_s}{(1-\overline{a}_pb_j)(1-\overline{b}_sb_j)}=\frac{\overline{a}_p}{1-\overline{a}_pb_j}-\frac{\overline{b}_s}{1-\overline{b}_sb_j},
\end{displaymath}
we get
\begin{displaymath}
\begin{split}
\sum_{j=1}^{n}&\frac{d_j}{(1-\overline{a}_pb_j)(1-\overline{b}_sb_j)}=\sum_{j=1}^{n}\frac{d_j}{\overline{a}_p-\overline{b}_s}\left( \frac{\overline{a}_p}{1-\overline{a}_pb_j}-\frac{\overline{b}_s}{1-\overline{b}_sb_j}\right)\\
=&\sum_{j=1}^{n}\frac{d_j}{\overline{a}_p-\overline{b}_s} \left( \frac{\overline{a}_p-\overline{b}_1}{(1-\overline{a}_pb_j)(1-\overline{b}_1b_j)}+\frac{\overline{b}_1-\overline{a}_1}{(1-\overline{a}_1b_j)(1-\overline{b}_1b_j)}\right.\\
&\left.+
\frac{\overline{a}_1-\overline{b}_s}{(1-\overline{a}_1b_j)(1-\overline{b}_sb_j)} \right )\\
=& \frac{\overline{a}_p-\overline{b}_1}{\overline{a}_p-\overline{b}_s}\sum_{j=1}^{n}\frac{d_j}{(1-\overline{a}_pb_j)(1-\overline{b}_1b_j)}+\frac{\overline{b}_1-\overline{a}_1}{\overline{a}_p-\overline{b}_s}\sum_{j=1}^{n}\frac{d_j}{(1-\overline{a}_1b_j)(1-\overline{b}_1b_j)}
\\&+\frac{\overline{a}_1-\overline{b}_s}{\overline{a}_p-\overline{b}_s}\sum_{j=1}^{n}\frac{d_j}{(1-\overline{a}_1b_j)(1-\overline{b}_sb_j)}.
\end{split}
\end{displaymath}
It follows that
\begin{displaymath}
\begin{split}
r_{s,p}&=\frac{\overline{c}_p}{\overline{a}_p-\overline{b}_s}\frac{\overline{\alpha'(a_p)}}{\overline{\beta'(b_s)}}\overline{\beta(a_p)}+\frac{1}{\overline{\beta'(b_s)}}\sum_{j=1}^{n}\frac{d_j}{(1-\overline{a}_pb_j)(1-\overline{b}_sb_j)}
\\&= \frac{\overline{c}_p}{\overline{a}_p-\overline{b}_s}\frac{\overline{\alpha'(a_p)}}{\overline{\beta'(b_s)}}\overline{\beta(a_p)}+\frac{\overline{a}_p-\overline{b}_1}{\overline{a}_p-\overline{b}_s}\frac{1}{\overline{\beta'(b_s)}}\sum_{j=1}^{n}\frac{d_j}{(1-\overline{a}_pb_j)(1-\overline{b}_1b_j)} \\&\quad- \frac{\overline{c}_1}{\overline{a}_p-\overline{b}_s}\frac{\overline{\alpha'(a_1)}}{\overline{\beta'(b_s)}}\overline{\beta(a_1)}+\frac{\overline{b}_1-\overline{a}_1}{\overline{a}_p-\overline{b}_s}\frac{1}{\overline{\beta'(b_s)}}\sum_{j=1}^{n}\frac{d_j}{(1-\overline{a}_1b_j)(1-\overline{b}_1b_j)}\\
&\quad+ \frac{\overline{c}_1}{\overline{a}_p-\overline{b}_s}\frac{\overline{\alpha'(a_1)}}{\overline{\beta'(b_s)}}\overline{\beta(a_1)}+\frac{\overline{a}_1-\overline{b}_s}{\overline{a}_p-\overline{b}_s}\frac{1}{\overline{\beta'(b_s)}}\sum_{j=1}^{n}\frac{d_j}{(1-\overline{a}_1b_j)(1-\overline{b}_sb_j)}\\
&=\frac{\overline{a}_p-\overline{b}_1}{\overline{a}_p-\overline{b}_s}\frac{\overline{\beta'(b_1)}}{\overline{\beta'(b_s)}}r_{1,p}+\frac{\overline{b}_1-\overline{a}_1}{\overline{a}_p-\overline{b}_s}\frac{\overline{\beta'(b_1)}}{\overline{\beta'(b_s)}}r_{1,1}+\frac{\overline{a}_1-\overline{b}_s}{\overline{a}_p-\overline{b}_s}r_{s,1}
\\ &=\frac{\overline{\beta'(b_s)}(\overline{a}_1-\overline{b}_s)r_{s,1}+\overline{\beta'(b_1)}(\overline{b}_1-\overline{a}_1)r_{1,1}+\overline{\beta'(b_1})(\overline{a}_p-\overline{b}_1)r_{1,p}}{\overline{\beta'(b_s)}(\overline{a}_p-\overline{b}_s)}.
\end{split}
\end{displaymath}

$(\mathrm{b})$ $l>0$.

\vspace{0.1cm}

In this case $a_p=b_p$ for $p\leq l$ and $a_p\neq b_s$ for $p>l$ and every $1\leq s\leq n$. Hence
\begin{displaymath}
\langle k_{a_p}^{\beta}, \widetilde{k}_{b_s}^{\beta}\rangle=\left\{\begin{array}{cc}
\overline{\beta'(b_s)}\delta_{s,p}& \mathrm{for}\quad p\leq l,\\
 \frac{\overline{\beta(a_p)}}{\overline{a}_p-\overline{b}_s} & \mathrm{for}\quad p>l,
\end{array}\right.
\end{displaymath}
and
\begin{displaymath}
r_{s,p}=\overline{c}_p\overline{\alpha'(a_p)}\delta_{s,p}+ \frac{1}{\overline{\beta'(b_s)}}\sum_{j=1}^{n}\frac{d_j}{(1-\overline{a}_pb_j)(1-\overline{b}_sb_j)}
\end{displaymath}
for $p\leq l$, $1\leq s\leq n$, and
\begin{displaymath}
r_{s,p}=\frac{\overline{c}_p}{\overline{a}_p-\overline{b}_s}\frac{\overline{\alpha'(a_p)}}{\overline{\beta'(b_s)}}\overline{\beta(a_p)}+ \frac{1}{\overline{\beta'(b_s)}}\sum_{j=1}^{n}\frac{d_j}{(1-\overline{a}_pb_j)(1-\overline{b}_sb_j)}
\end{displaymath}
for $p>l$, $1\leq s\leq n$.

We now show that $r_{s,p}$ satisfies \eqref{3} for all $p,s$ such that $1\leq s\leq n$, $p>l$ or $1\leq s\leq n$, $1\leq p\leq l$, $p\neq s$.

Clearly, \eqref{3} holds for all $p,s$ with $p\neq s$ and such that $s=1$ or $p=1$. Assume that $s\neq1$ and $p\neq 1$. If $p>l$, then using the fact that $a_1=b_1$, we get
\begin{displaymath}
\begin{split}
r_{s,p}&=\frac{\overline{c}_p}{\overline{a}_p-\overline{b}_s}\frac{\overline{\alpha'(a_p)}}{\overline{\beta'(b_s)}}\overline{\beta(a_p)}+\frac{1}{\overline{\beta'(b_s)}}\sum_{j=1}^{n}\frac{d_j}{(1-\overline{a}_pb_j)(1-\overline{b}_sb_j)}
\\ &= \frac{\overline{c}_p}{\overline{a}_p-\overline{b}_s}\frac{\overline{\alpha'(a_p)}}{\overline{\beta'(b_s)}}\overline{\beta(a_p)}+\frac{1}{\overline{\beta'(b_s)}}\sum_{j=1}^{n}\frac{d_j}{\overline{a}_p-\overline{b}_s}\left( \frac{\overline{a}_p-\overline{b}_1}{(1-\overline{a}_pb_j)(1-\overline{b}_1b_j)}\right.\\
&\quad+\left.\frac{\overline{a}_1-\overline{b}_s}{(1-\overline{a}_1b_j)(1-\overline{b}_sb_j)}\right)\\
&=\frac{\overline{a}_p-\overline{b}_1}{\overline{a}_p-\overline{b}_s}\frac{\overline{\beta'(b_1)}}{\overline{\beta'(b_s)}}\left(\frac{\overline{c}_p}{\overline{a}_p-\overline{b}_1}\frac{\overline{\alpha'(a_p)}}{\overline{\beta'(b_1)}}\overline{\beta(a_p)}+\frac{1}{\overline{\beta'(b_1)}}\sum_{j=1}^{n}\frac{d_j}{(1-\overline{a}_pb_j)(1-\overline{b}_1b_j)}\right)\\
&\quad+\frac{\overline{a}_1-\overline{b}_s}{\overline{a}_p-\overline{b}_s}\frac{1}{\overline{\beta'(b_s)}}\sum_{j=1}^{n}\frac{d_j}{(1-\overline{a}_1b_j)(1-\overline{b}_sb_j)}\\
&=\frac{\overline{a}_p-\overline{b}_1}{\overline{a}_p-\overline{b}_s}\frac{\overline{\beta'(b_1)}}{\overline{\beta'(b_s)}}r_{1,p}+\frac{\overline{a}_1-\overline{b}_s}{\overline{a}_p-\overline{b}_s}r_{s,1}=\frac{\overline{\beta'(b_s)}(\overline{a}_1-\overline{b}_s)r_{s,1}+\overline{\beta'(b_1)}(\overline{a}_p-\overline{b}_1)r_{1,p}}{\overline{\beta'(b_s)}(\overline{a}_p-\overline{b}_s)}.
\end{split}
\end{displaymath}
Similarly, if $p\leq l$, $s\neq p$, then
\begin{displaymath}
\begin{split}
r_{s,p}&=\frac{1}{\overline{\beta'(b_s)}}\sum_{j=1}^{n}\frac{d_j}{(1-\overline{a}_pb_j)(1-\overline{b}_sb_j)}
\\&=\frac{\overline{a}_p-\overline{b}_1}{\overline{a}_p-\overline{b}_s}\frac{\overline{\beta'(b_1)}}{\overline{\beta'(b_s)}}\frac{1}{\overline{\beta'(b_1)}}\sum_{j=1}^{n}\frac{d_j}{(1-\overline{a}_pb_j)(1-\overline{b}_1b_j)}\\
&\quad+\frac{\overline{a}_1-\overline{b}_s}{\overline{a}_p-\overline{b}_s}\frac{1}{\overline{\beta'(b_s)}}\sum_{j=1}^{n}\frac{d_j}{(1-\overline{a}_1b_j)(1-\overline{b}_sb_j)}\\
&=\frac{\overline{a}_p-\overline{b}_1}{\overline{a}_p-\overline{b}_s}\frac{\overline{\beta'(b_1)}}{\overline{\beta'(b_s)}}r_{1,p}+\frac{\overline{a}_1-\overline{b}_s}{\overline{a}_p-\overline{b}_s}r_{s,1}\\
&=\frac{\overline{\beta'(b_s)}(\overline{a}_1-\overline{b}_s)r_{s,1}+\overline{\beta'(b_1)}(\overline{a}_p-\overline{b}_1)r_{1,p}}{\overline{\beta'(b_s)}(\overline{a}_p-\overline{b}_s)}.
\end{split}
\end{displaymath}

In particular, \eqref{3} holds for all $p,s$ such that $l<s\leq n$, $1\leq p\leq m$.

To complete this part of the proof we need to show that $r_{s,p}$ satisfies \eqref{2} for all $p,s$ such that $1\leq s\leq l$, $1\leq p\leq m$, $s\neq p$. Again, this is obvious for $s=1$. If $1< s\leq l$,  then
\begin{displaymath}
\begin{split}
r_{s,1}&=\frac{1}{\overline{\beta'(b_s)}}\sum_{j=1}^{n}\frac{d_j}{(1-\overline{a}_1b_j)(1-\overline{b}_sb_j)}\\
&=\frac{1}{\overline{\beta'(b_s)}}\sum_{j=1}^{n}\frac{d_j}{(1-\overline{a}_sb_j)(1-\overline{b}_1b_j)}=\frac{\overline{\beta'(b_1)}}{\overline{\beta'(b_s)}}r_{1,s}
\end{split}
\end{displaymath}
and \eqref{2} follows from \eqref{3}.

The proof of necessity of the given conditions (for $l=0$ and for $l>0$) is now complete. We now prove sufficiency.

Assume that $l=0$ and note that the linear space $V$ of all the matrices satisfying \eqref{1} has dimension $m+n-1$. By the first part of the proof, the set $V_0$ of all the matrices representing operators from $\mathscr{T}(\alpha,\beta)$  is a subspace of $V$,
\begin{displaymath}
V_0=\{M_{A_{\varphi}^{\alpha, \beta}}\colon\ A_{\varphi}^{\alpha, \beta}\in \mathscr{T}(\alpha,\beta)\}\subset V.
\end{displaymath}
However, by Proposition \ref{dym} we know that $V_0$ also has dimension $m+n-1$, and so $V_0=V$.

The proof for $l>0$ is analogous.
\end{proof}

\begin{rem}
\begin{itemize}
\item[(a)] Theorem \ref{macierz} states that if $\alpha$ and $\beta$ have no common zeros ($l=0$), then the matrix representing an operator from $\mathcal{T}(\alpha,\beta)$ is determined by the entries along the first row and the first column. A slight modification of the proof shows that one can in fact take any other row and any other column.
\item[(b)] Note that in the proof of part $(\mathrm{b})$ of Theorem \ref{macierz} we actually showed that the elements $r_{s,p}$ satisfy
\begin{equation*}
r_{s,p}=\frac{\overline{\beta'(b_s)}(\overline{a}_1-\overline{b}_s)r_{s,1}+\overline{\beta'(b_1)}(\overline{a}_p-\overline{b}_1)r_{1,p}}{\overline{\beta'(b_s)}(\overline{a}_p-\overline{b}_s)}
\end{equation*}
for all $p,s$ such that $1\leq s\leq n$, $p>l$ or $1\leq s\leq n$, $1\leq p\leq l$, $p\neq s$. However, this only says that the matrix representing an asymmetric truncated Toeplitz operator is determined by $m+n+l-2$ of its entries, which is more that $m+n-1$ for $l>1$. To reduce the number of the determining entries we consider two equations: \eqref{2} and \eqref{3}. These equations say that the matrix is determined by entries along the first row, first $l$ entries along the main diagonal and last $n-l$ entries along the first column.
\item[(c)] A modification of the proof of part $(\mathrm{b})$ of Theorem \ref{macierz} shows that the first row and column can be replaced by any other row and column that intersect at one of the first $l$ elements of the main diagonal. The theorem can also be formulated with rows in place of the columns and vice versa.
\item[(d)] Note that if $\alpha=\beta$ is a Blaschke product with $m$ distinct zeros, then $l=m=n$ and part (b) of Theorem \ref{macierz} is precisely the result obtained in \cite[Thm. 1.4]{w}.
\end{itemize}
\end{rem}

Theorem \ref{macierz} can also be formulated in terms of the matrix representation with respect to $\{\widetilde{k}_{a_1}^{\alpha},\ldots, \widetilde{k}_{a_m}^{\alpha} \}$ and $\{\widetilde{k}_{b_1}^{\beta},\ldots, \widetilde{k}_{b_n}^{\beta} \}$.

\begin{thm}
Let the function $\alpha$ be a finite Blaschke product with $m$ distinct zeros $a_1,\ldots, a_m$, let $\beta$ be a finite Blaschke product with $n$ distinct zeros $b_1,\ldots, b_n$ and assume that $\alpha$ and $\beta$ have precisely $l$ zeros in common: $a_i=b_i$ for $i\leq l$ ($l=0$ if there are no zeros in common). Let $A$ be any linear transformation from $K_{\alpha}$ into $K_{\beta}$. If $\widetilde{M}_A=(t_{s,p})$ is the matrix representation of $A$ with respect to the bases
$\{\widetilde{k}_{a_1}^{\alpha},\ldots, \widetilde{k}_{a_m}^{\alpha} \}$ and $\{\widetilde{k}_{b_1}^{\beta},\ldots, \widetilde{k}_{b_n}^{\beta} \}$, and
  \vspace{0.2cm}
\begin{itemize}
\item[(a)]
$l=0 $, then $A\in \mathscr{T}(\alpha,\beta)$ if and only if
\begin{equation}\label{5}
t_{s,p}=\frac{\beta'(b_s)({a_1}-{b_s})t_{s,1}+{\beta'(b_1)}({b_1}-{a_1})t_{1,1}+\beta'(b_1)({a_p}-{b_1})t_{1,p}}{{\beta'(b_s)}({a_p}-{b_s})}
\end{equation} for all $1\leq  p\leq m$ and $1\leq s \leq n$;
  \vspace{0.5cm}
\item[(b)] $l>0$, then $A\in \mathscr{T}(\alpha,\beta)$ if and only if
\begin{equation*}%\label{6}
t_{s,p}=\frac{\beta'(b_1)({a_1}-{b_s})t_{1,s}+\beta'(b_1)({a_p}-{b_1})t_{1,p}}{{\beta'(b_s)(a_p-b_s)}}
\end{equation*} for all $p,s$ such that $1\leq p \leq m$, $1\leq  s\leq l$, $s\neq p$, and
\begin{equation*}%\label{7}
t_{s,p}=\frac{\beta'(b_s)({a_1}-{b_s})t_{s,1}+\beta'(b_1)({a_p}-{b_1})t_{1,p}}{{\beta'(b_s)}({a_p}-{b_s})}
\end{equation*} for all $p,s$ such that  $1\leq p \leq m$, $l<  s\leq n$.
\end{itemize}
\end{thm}
\begin{proof}
Assume first that $l=0$. Let $A$ be any linear transformation from $K_{\alpha}$ into $K_{\beta}$ and let $\widetilde{M}_A=(t_{s,p})$ be the matrix representation of $A$ with respect to $\{\widetilde{k}_{a_1}^{\alpha},\ldots, \widetilde{k}_{a_m}^{\alpha} \}$ and $\{\widetilde{k}_{b_1}^{\beta},\ldots, \widetilde{k}_{b_n}^{\beta} \}$. By \cite[Lem. 3.2]{ptak}, $A\in \mathscr{T}(\alpha,\beta)$ if and only if  $A^{*}\in \mathscr{T}(\beta,\alpha)$. Theorem \ref{macierz}(a) implies that the latter holds if and only if
\begin{equation}\label{8}
r_{s,p}=\frac{\overline{\alpha'(a_s)}(\overline{b}_1-\overline{a}_s)r_{s,1}+\overline{\alpha'(a_1)}(\overline{a}_1-\overline{b}_1)r_{1,1}+\overline{\alpha'(a_1})(\overline{b}_p-\overline{a}_1)r_{1,p}}{\overline{\alpha'(a_s)}(\overline{b}_p-\overline{a}_s)}
\end{equation} for all $1\leq s \leq m$ and $1\leq  p\leq n$, where $M_{A^{*}}=(r_{s,p})$ is the matrix representation of $A^{*}$ with respect to $\{k_{b_1}^{\beta},\ldots, k_{b_n}^{\beta} \}$ and $\{k_{a_1}^{\alpha},\ldots, k_{a_m}^{\alpha} \}$.

We show that $M_{A^{*}}$ satisfies \eqref{8} if and only if $\widetilde{M}_A$ satisfies \eqref{5}. To this end, we note that
\begin{displaymath}
\begin{split}
t_{s,p}&= {\beta'(b_s)}^{-1}\langle A \widetilde{k}_{a_p}^{\alpha}, {k}_{b_s}^{\beta}\rangle \\&= {\beta'(b_s)}^{-1}\overline{\langle A^{*}k_{b_s}^{\beta}, \widetilde{k}_{a_p}^{\alpha}\rangle}=\frac{\alpha'(a_p)}{\beta'(b_s)}\overline{r}_{p,s} .
\end{split}
\end{displaymath}
It follows that if $M_{A^{*}}$ satisfies \eqref{8}, then
\begin{displaymath}
\begin{split}
t_{s,p}&=\frac{\alpha'(a_p)}{\beta'(b_s)}\frac{\alpha'(a_1)({b_s}-{a_1})\overline{r}_{1,s}+{\alpha'(a_1)}({a_1}-{b_1})\overline{r}_{1,1}+{\alpha'(a_p)}({b_1}-{a_p})\overline{r}_{p,1}}{{\alpha'(a_p)}({b_s}-{a_p})}\\ &=
\frac{({a_1}-{b_s})\alpha'(a_1)\overline{r}_{1,s}
+({b_1}-{a_1})\alpha'(a_1)\overline{r}_{1,1}
+({a_p}-{b_1})\alpha'(a_p)\overline{r}_{p,1}}{{\beta'(b_s)}({a_p}-{b_s})}\\ &=\frac{\beta'(b_s)({a_1}-{b_s})t_{s,1}+\beta'(b_1)({b_1}-{a_1})t_{1,1}+\beta'(b_1)({a_p}-{b_1})t_{1,p}}{{\beta'(b_s)}({a_p}-{b_s})}.
\end{split}
\end{displaymath}

Similarly, if $\widetilde{M}_A$ satisfies \eqref{5}, then $M_{A^{*}}$ satisfies \eqref{8}. This completes the proof for the case $l=0$.

The other case is left to the reader.
\end{proof}

\subsection{Clark bases and modified Clark bases}

Now let $\alpha$ and $\beta$ be as in \eqref{blaszki} but do not assume that the zeros are distinct.

For any $\lambda_1\in\partial\mathbb{D}$ define
$$\alpha_{\lambda_1}=\frac{\lambda_1+\alpha(0)}{1+\overline{\alpha(0)}\lambda_1}.$$
Then $\alpha_{\lambda_1}\in\partial\mathbb{D}$ and, since $|\alpha|=1$ and $|\alpha'|>0$ on $\partial\mathbb{D}$ (see \cite[p. 6]{gar}), the equation
\begin{equation}\label{alf}
\alpha(\eta)=\alpha_{\lambda_1}
\end{equation}
has precisely $m$ distinct solutions $\eta_1,\ldots,\eta_m$, which lie on the unit circle $\partial\mathbb{D}$. The corresponding kernel functions $k_{\eta_1}^{\alpha},\ldots,k_{\eta_m}^{\alpha}$ belong to $K_{\alpha}$. Moreover,
\begin{equation*}
\langle k_{\eta_i}^{\alpha},k_{\eta_j}^{\alpha}\rangle=\left\{\begin{array}{cc}\|k_{\eta_i}^{\alpha}\|^2&\mathrm{for}\ i=j,\\0&\mathrm{for}\ i\neq j.\end{array}\right.
\end{equation*}
Therefore the functions $k_{\eta_1}^{\alpha},\ldots,k_{\eta_m}^{\alpha}$ form an orthogonal basis for $K_{\alpha}$ and the normalized kernel functions
\begin{equation*}
v_{\eta_j}^{\alpha}=\|k_{\eta_j}^{\alpha}\|^{-1}k_{\eta_j}^{\alpha},\quad j=1,\ldots,m,
\end{equation*}
 form an orthonormal basis for $K_{\alpha}$. The basis $\{v_{\eta_1}^{\alpha},\ldots,v_{\eta_m}^{\alpha}\}$ is called the Clark basis corresponding to $\lambda_1$ (see \cite{c} and \cite{gar} for more details).

 We can also define the so-called modified Clark basis corresponding to $\lambda_1$ by
\begin{equation*}
e_{\eta_j}^{\alpha}=\omega_{j}^{\alpha}v_{\eta_j}^{\alpha},\quad j=1,\ldots,m,
\end{equation*}
where
\begin{equation*}
\omega_{j}^{\alpha}=e^{-\frac{i}{2}(\mathrm{arg}\eta_j-\mathrm{arg}\lambda_1)},\quad j=1,\ldots,m.
\end{equation*}
Then the basis $\{e_{\eta_1}^{\alpha},\ldots,e_{\eta_m}^{\alpha}\}$ is an orthonormal basis for $K_{\alpha}$ and such that $$C_{\alpha}e_{\eta_j}^{\alpha}=e_{\eta_j}^{\alpha},\quad j=1,\ldots,m,$$
where $C_{\alpha}$ is the conjugation given by \eqref{numerek3}.

Similarly, for any $\lambda_2\in\partial\mathbb{D}$ there are precisely $n$ distinct solutions $\zeta_1,\ldots,\zeta_n$ on $\partial\mathbb{D}$ of the equation
\begin{equation}\label{bet}
\beta(\zeta)=\beta_{\lambda_2}=\frac{\lambda_2+\beta(0)}{1+\overline{\beta(0)}\lambda_2}.
\end{equation}
The Clark basis $\{v_{\zeta_1}^{\beta},\ldots,v_{\zeta_n}^{\beta}\}$ and modified Clark basis $\{e_{\zeta_1}^{\beta},\ldots,e_{\zeta_n}^{\beta}\}$ corresponding to $\lambda_2$ are defined as above by
\begin{equation*}
v_{\zeta_j}^{\beta}=\|k_{\zeta_j}^{\beta}\|^{-1}k_{\zeta_j}^{\beta},\quad j=1,\ldots,n,
\end{equation*}
and
\begin{equation*}
e_{\zeta_j}^{\beta}=\omega_{j}^{\beta}v_{\zeta_j}^{\beta},\quad j=1,\ldots,n,
\end{equation*}
where
\begin{equation*}
\omega_{j}^{\beta}=e^{-\frac{i}{2}(\mathrm{arg}\zeta_j-\mathrm{arg}\lambda_2)},\quad j=1,\ldots,n.
\end{equation*}

Of course, it may happen that the equations \eqref{alf} and \eqref{bet} have some solutions in common. Here we assume that \eqref{alf} and \eqref{bet} have precisely $l$ solutions in common ($l=0$ if there are no solutions in common), these solutions being $\eta_j=\zeta_j$ for $j\leq l$.

\begin{thm}\label{macierz3}
Let $\alpha$ and $\beta$ be two finite Blaschke products of degree $m>0$ and $n>0$, respectively. Let $\{v_{\eta_1}^{\alpha},\ldots, v_{\eta_m}^{\alpha} \}$ be the Clark basis for $K_{\alpha}$ corresponding to $\lambda_1\in\partial\mathbb{D}$, let $\{v_{\zeta_1}^{\beta},\ldots, v_{\zeta_n}^{\beta} \}$ be the Clark basis for $K_{\beta}$ corresponding to $\lambda_2\in\partial\mathbb{D}$ and assume that the sets $\{\eta_1,\ldots,\eta_m\}$, $\{\zeta_1,\ldots,\zeta_n\}$ have precisely $l$ elements in common: $\eta_j=\zeta_j$ for $j\leq l$ ($l=0$ if there are no elements in common). Finally, let $A$ be any linear transformation from $K_{\alpha}$ into $K_{\beta}$. If $M_A=(r_{s,p})$ is the matrix representation of $A$ with respect to the bases $\{v_{\eta_1}^{\alpha},\ldots, v_{\eta_m}^{\alpha} \}$ and $\{v_{\zeta_1}^{\beta},\ldots, v_{\zeta_n}^{\beta} \}$, and
  \vspace{0.2cm}
\begin{itemize}
\item[(a)]$l=0 $, then $A\in \mathscr{T}(\alpha,\beta)$ if and only if
\begin{equation}\label{c1}
\begin{split}
r_{s,p}=&\left(\frac{\sqrt{|\alpha'(\eta_1)|}}{\sqrt{|\alpha'(\eta_p)|}}\frac{\eta_p}{\eta_1}\frac{\eta_1-\zeta_s}{\eta_p-\zeta_s}r_{s,1}+\frac{\sqrt{|\alpha'(\eta_1)|}\sqrt{|\beta'(\zeta_1)|}}{\sqrt{|\alpha'(\eta_p)|}\sqrt{|\beta'(\zeta_s)|}}\frac{\eta_p}{\eta_1}\frac{\zeta_1-\eta_1}{\eta_p-\zeta_s}r_{1,1}\right.
\\&\phantom{=\frac1{\sqrt{|\alpha'(\eta_p)|}}}+\left.\frac{\sqrt{|\beta'(\zeta_1)|}}{\sqrt{|\beta'(\zeta_s)|}}\frac{\eta_p-\zeta_1}{\eta_p-\zeta_s}r_{1,p}\right)
\end{split}
\end{equation} for all $1\leq p \leq m$ and $1\leq  s\leq n$;
\vspace{0.5cm}
\item[(b)]$l>0$, then $A\in \mathscr{T}(\alpha,\beta)$ if and only if
\begin{equation}\label{c2}
r_{s,p}=\left(\frac{\sqrt{|\alpha'(\eta_s)|}\sqrt{|\beta'(\zeta_1)|}}{\sqrt{|\alpha'(\eta_p)|}\sqrt{|\beta'(\zeta_s)|}}\frac{\eta_p}{\eta_s}\frac{\eta_1-\zeta_s}{\eta_p-\zeta_s}r_{1,s}+\frac{\sqrt{|\beta'(\zeta_1)|}}{\sqrt{|\beta'(\zeta_s)|}}\frac{\eta_p-\zeta_1}{\eta_p-\zeta_s}r_{1,p}\right)
\end{equation} for all $p,s$ such that $1\leq p \leq m$, $1\leq  s\leq l$, $s\neq p$, and
\begin{equation}\label{c3}
r_{s,p}=\left(\frac{\sqrt{|\alpha'(\eta_1)|}}{\sqrt{|\alpha'(\eta_p)|}}\frac{\eta_p}{\eta_1}\frac{\eta_1-\zeta_s}{\eta_p-\zeta_s}r_{s,1}+\frac{\sqrt{|\beta'(\zeta_1)|}}{\sqrt{|\beta'(\zeta_s)|}}\frac{\eta_p-\zeta_1}{\eta_p-\zeta_s}r_{1,p}\right)
\end{equation} for all $p,s$ such that $1\leq p \leq m$, $l<  s\leq n$.
\end{itemize}
\end{thm}
\begin{proof}
Let $A$ be any linear transformation from $K_{\alpha}$ into $K_{\beta}$ and let $M_A=(r_{s,p})$ be its matrix representation with respect to the bases $\{v_{\eta_1}^{\alpha},\ldots, v_{\eta_m}^{\alpha} \}$ and $\{v_{\zeta_1}^{\beta},\ldots, v_{\zeta_n}^{\beta} \}$.

We first show that if $A$ belongs to $\mathcal{T}(\alpha,\beta)$, then $M_A$ has the desired properties.

Assume that $A=A_{\varphi}^{\alpha,\beta}$ for $\varphi\in L^2(\partial\mathbb{D})$. To compute $r_{s,p}$ pick $m+n-1$ distinct points $\xi_1,\ldots,\xi_{m+n-1}$ from $\partial\mathbb{D}$, different from $\eta_i$, $i=1,\ldots,m$, and from $\zeta_j$, $j=1,\ldots,n$. It follows form Corollary \ref{bazy} that the operators $k_{\xi_i}^{\beta}\otimes k_{\xi_i}^{\alpha}$, $i=1,\ldots,m+n-1$, form a basis for $\mathcal{T}(\alpha,\beta)$. Hence there exist scalars $c_1,\ldots,c_{m+n-1}$ such that
\begin{equation}\label{suma}
A_{\varphi}^{\alpha,\beta}=\sum_{i=1}^{m+n-1}c_i k_{\xi_i}^{\beta}\otimes k_{\xi_i}^{\alpha}.
\end{equation}
Since the Clark bases are orthonormal,
$$r_{s,p}=\langle A_{\varphi}^{\alpha,\beta}v_{\eta_p}^{\alpha},v_{\zeta_s}^{\beta}\rangle$$
for $1\leq p\leq m$  and $1\leq s\leq n$. We now compute $r_{s,p}$ in terms of $c_i$, $i=1,\ldots,m+n-1$. By \eqref{suma} we have
\begin{displaymath}
A_{\varphi}^{\alpha,\beta}v_{\eta_p}^{\alpha}=\sum_{i=1}^{m+n-1}c_i k_{\xi_i}^{\beta}\otimes k_{\xi_i}^{\alpha}(v_{\eta_p}^{\alpha})=\sum_{i=1}^{m+n-1}c_i v_{\eta_p}^{\alpha}(\xi_i)k_{\xi_i}^{\beta},
\end{displaymath}
and
\begin{displaymath}
\begin{split}
r_{s,p}&=\sum_{i=1}^{m+n-1}c_i v_{\eta_p}^{\alpha}(\xi_i)\overline{v_{\zeta_s}^{\beta}(\xi_i)}\\
&=\frac1{\|k_{\eta_p}^{\alpha}\|\|k_{\zeta_s}^{\beta}\|}\sum_{i=1}^{m+n-1}c_i\frac{1-\overline{\alpha(\eta_p)}\alpha(\xi_i)}{1-\overline{\eta}_p\xi_i}\frac{1-\beta(\zeta_s)\overline{\beta(\xi_i)}}{1-\zeta_s\overline{\xi}_i}\\
&=\frac{\eta_p}{\|k_{\eta_p}^{\alpha}\|\|k_{\zeta_s}^{\beta}\|}\sum_{i=1}^{m+n-1}c_i\xi_i\frac{(\overline{\alpha(\eta_p)}\alpha(\xi_i)-1)(1-\beta(\zeta_s)\overline{\beta(\xi_i)})}{(\xi_i-\eta_p)(\xi_i-\zeta_s)}.
\end{split}
\end{displaymath}
Recall that $\|k_{\eta_p}^{\alpha}\|=\sqrt{|\alpha'(\eta_p)|}$ and $\|k_{\zeta_s}^{\beta}\|=\sqrt{|\beta'(\zeta_s)|}$. Moreover, $\eta_p$ and $\zeta_s$ are solutions of \eqref{alf} and \eqref{bet}, respectively. Consequently, $\alpha(\eta_p)=\alpha_{\lambda_1}$, $p=1,\ldots,m$, $\beta(\zeta_s)=\beta_{\lambda_2}$, $s=1,\ldots,n$, and
\begin{equation}\label{er}
r_{s,p}=\frac{\eta_p}{\sqrt{|\alpha'(\eta_p)|}\sqrt{|\beta'(\zeta_s)|}}\sum_{i=1}^{m+n-1}\frac{d_i}{(\xi_i-\eta_p)(\xi_i-\zeta_s)},
\end{equation}
where
\begin{displaymath}
\begin{split}
d_i&=c_i\xi_i(\overline{\alpha(\eta_p)}\alpha(\xi_i)-1)(1-\beta(\zeta_s)\overline{\beta(\xi_i)})\\
&=c_i\xi_i(\overline{\alpha_{\lambda_1}}\alpha(\xi_i)-1)(1-\beta_{\lambda_2}\overline{\beta(\xi_i)})
\end{split}
\end{displaymath}
is independent of $p$ and $s$.

\vspace{0.1cm}

$(\mathrm{a})$ $l=0$.

\vspace{0.1cm}

In this case $\eta_p\neq\zeta_s$ for all $1\leq p\leq m$ and $1\leq s\leq n$. Using \eqref{er} and the equality
\begin{equation*}
\frac{\eta_p-\zeta_s}{(\xi_i-\eta_p)(\xi_i-\zeta_s)}=\frac1{\xi_i-\eta_p}-\frac1{\xi_i-\zeta_s}
\end{equation*}
we get
\begin{displaymath}
\begin{split}
r_{s,p}&=\frac{\eta_p}{\sqrt{|\alpha'(\eta_p)|}\sqrt{|\beta'(\zeta_s)|}}\sum_{i=1}^{m+n-1}\frac{d_i}{(\xi_i-\eta_p)(\xi_i-\zeta_s)}\\
&=\frac1{\sqrt{|\alpha'(\eta_p)|}\sqrt{|\beta'(\zeta_s)|}}\frac{\eta_p}{\eta_p-\zeta_s}\sum_{i=1}^{m+n-1}d_i\left(\frac1{\xi_i-\eta_p}-\frac1{\xi_i-\zeta_s}\right)\\
&=\frac1{\sqrt{|\alpha'(\eta_p)|}\sqrt{|\beta'(\zeta_s)|}}\frac{\eta_p}{\eta_p-\zeta_s}\sum_{i=1}^{m+n-1}d_i\left(\frac{\eta_p-\zeta_1}{(\xi_i-\eta_p)(\xi_i-\zeta_1)}\right.\\
&\quad+\left.\frac{\zeta_1-\eta_1}{(\xi_i-\zeta_1)(\xi_i-\eta_1)}+\frac{\eta_1-\zeta_s}{(\xi_i-\eta_1)(\xi_i-\zeta_s)}\right)\\
&=\frac1{\sqrt{|\alpha'(\eta_p)|}\sqrt{|\beta'(\zeta_s)|}}\frac{\eta_p}{\eta_p-\zeta_s}\left( \sqrt{|\alpha'(\eta_p)|}\sqrt{|\beta'(\zeta_1)|}\frac{\eta_p-\zeta_1}{\eta_p}r_{1,p}\right.\\
&\quad+\left.\sqrt{|\alpha'(\eta_1)|}\sqrt{|\beta'(\zeta_1)|}\frac{\zeta_1-\eta_1}{\eta_1}r_{1,1}+\sqrt{|\alpha'(\eta_1)|}\sqrt{|\beta'(\zeta_s)|}\frac{\eta_1-\zeta_s}{\eta_1}r_{s,1}\right)\\
&=\left(\frac{\sqrt{|\alpha'(\eta_1)|}}{\sqrt{|\alpha'(\eta_p)|}}\frac{\eta_p}{\eta_1}\frac{\eta_1-\zeta_s}{\eta_p-\zeta_s}r_{s,1}+\frac{\sqrt{|\alpha'(\eta_1)|}\sqrt{|\beta'(\zeta_1)|}}{\sqrt{|\alpha'(\eta_p)|}\sqrt{|\beta'(\zeta_s)|}}\frac{\eta_p}{\eta_1}\frac{\zeta_1-\eta_1}{\eta_p-\zeta_s}r_{1,1}\right.
\\&\quad+\left.\frac{\sqrt{|\beta'(\zeta_1)|}}{\sqrt{|\beta'(\zeta_s)|}}\frac{\eta_p-\zeta_1}{\eta_p-\zeta_s}r_{1,p}\right).
\end{split}
\end{displaymath}

Hence \eqref{c1} holds.

\vspace{0.1cm}

$(\mathrm{b})$ $l>0$.

\vspace{0.1cm}

In this case the proof of part $(\mathrm{a})$ can be repeated to show that \eqref{c1} holds for all $p,s$ such that $1\leq p\leq m$, $s>l$, and for all $p,s$ such that $1\leq p\leq m$, $1\leq s\leq l$, $s\neq p$. Since here $\zeta_1=\eta_1$, we get \eqref{c3}. %In particular, \eqref{c3} holds for $1\leq p\leq m$ and $s>l$.

We now show that \eqref{c2} holds for all $p,s$ such that $1\leq p\leq m$, $1\leq s\leq l$, $p\neq s$. Clearly, \eqref{c2} holds for $s=1$, $p\neq s$. Since here $\eta_s=\zeta_s$, it follows that for $1<s\leq l$,
\begin{displaymath}
\begin{split}
r_{s,1}&=\frac{\eta_1}{\sqrt{|\alpha'(\eta_1)|}\sqrt{|\beta'(\zeta_s)|}}\sum_{i=1}^{m+n-1}\frac{d_i}{(\xi_i-\eta_1)(\xi_i-\zeta_s)}\\
&=\frac{\eta_1}{\sqrt{|\alpha'(\eta_1)|}\sqrt{|\beta'(\zeta_s)|}}\sum_{i=1}^{m+n-1}\frac{d_i}{(\xi_i-\eta_s)(\xi_i-\zeta_1)}\\
&=\frac{\sqrt{|\alpha'(\eta_s)|}\sqrt{|\beta'(\zeta_1)|}}{\sqrt{|\alpha'(\eta_1)|}\sqrt{|\beta'(\zeta_s)|}}\frac{\eta_1}{\eta_s}r_{1,s}.
\end{split}
\end{displaymath}
Hence, for $1<s\leq l$, $p\neq s$, the equation \eqref{c2} follows form \eqref{c3}.

To complete the proof note that a matrix satisfying \eqref{c1} (or \eqref{c2} and \eqref{c3} in the case of $l>0$) is determined by $m+n-1$ entries and so the linear space $V$ of all such matrices has dimension $m+n-1$. As in the proof of Theorem \ref{macierz},
\begin{displaymath}
V_0=\{M_{A_{\varphi}^{\alpha, \beta}},A_{\varphi}^{\alpha, \beta}\in \mathscr{T}(\alpha,\beta)\}\subset V,
\end{displaymath}
and, since $V_0$ also has dimension $m+n-1$, we get $V_0=V$.
\end{proof}

\begin{rem}
\begin{itemize}
\item[(a)] Part $(\mathrm{a})$ of Theorem \ref{macierz3} states that the matrix representing an operator from $\mathcal{T}(\alpha,\beta)$ is determined by entries along the first row and the first column. The proof can be modified to show that one may replace the first row and the first column by any other row and any other column.
\item[(b)] Proof of part $(\mathrm{b})$ of Theorem \ref{macierz3} can also be modified to show that the first row and column can be replaced by any other row and column that intersect at one of the first $l$ elements of the main diagonal. The theorem can be formulated with rows in place of the columns and vice versa.
\item[(c)] If $\alpha=\beta$ is a Blaschke product of degeree $m$, then $l=m=n$. Moreover, if $\lambda_1=\lambda_2$, then $\eta_j=\zeta_j$ for all $j=1,\ldots, m$, and part (b) of Theorem \ref{macierz3} is precisely the result form \cite[Thm. 1.11]{w}.
\end{itemize}
\end{rem}

\begin{thm}%\label{macierz3}
Let $\alpha$ and $\beta$ be two finite Blaschke products of degree $m>0$ and $n>0$, respectively. Let $\{e_{\eta_1}^{\alpha},\ldots, e_{\eta_m}^{\alpha} \}$ be the modified Clark basis for $K_{\alpha}$ corresponding to $\lambda_1\in\partial\mathbb{D}$, let $\{e_{\zeta_1}^{\beta},\ldots, e_{\zeta_n}^{\beta} \}$ be the modified Clark basis for $K_{\beta}$ corresponding to $\lambda_2\in\partial\mathbb{D}$ and assume that the sets $\{\eta_1,\ldots,\eta_m\}$, $\{\zeta_1,\ldots,\zeta_n\}$ have precisely $l$ elements in common: $\eta_j=\zeta_j$ for $j\leq l$ ($l=0$ if there are no elements in common). Finally, let $A$ be any linear transformation from $K_{\alpha}$ into $K_{\beta}$. If $\widetilde{M}_A=(t_{s,p})$ is the matrix representation of $A$ with respect to the bases $\{e_{\eta_1}^{\alpha},\ldots, e_{\eta_m}^{\alpha} \}$ and $\{e_{\zeta_1}^{\beta},\ldots, e_{\zeta_n}^{\beta} \}$, and
  \vspace{0.2cm}
\begin{itemize}
\item[(a)] $l=0 $, then $A\in \mathscr{T}(\alpha,\beta)$ if and only if
\begin{equation}\label{c4}
\begin{split}
t_{s,p}=&\left(\frac{\sqrt{|\alpha'(\eta_1)|}}{\sqrt{|\alpha'(\eta_p)|}}\frac{\omega_1^{\alpha}}{\omega_p^{\alpha}}\frac{\eta_1-\zeta_s}{\eta_p-\zeta_s}t_{s,1}+\frac{\sqrt{|\alpha'(\eta_1)|}\sqrt{|\beta'(\zeta_1)|}}{\sqrt{|\alpha'(\eta_p)|}\sqrt{|\beta'(\zeta_s)|}}\frac{\omega_1^{\alpha}\omega_1^{\beta}}{\omega_p^{\alpha}\omega_s^{\beta}}\frac{\zeta_1-\eta_1}{\eta_p-\zeta_s}t_{1,1}\right.
\\&\phantom{=\frac1{\sqrt{|\alpha'(\eta_p)|}}}+\left.\frac{\sqrt{|\beta'(\zeta_1)|}}{\sqrt{|\beta'(\zeta_s)|}}\frac{\omega_1^{\beta}}{\omega_s^{\beta}}\frac{\eta_p-\zeta_1}{\eta_p-\zeta_s}t_{1,p}\right)
\end{split}
\end{equation} for all $1\leq p \leq m$ and $1\leq  s\leq n$;
\vspace{0.5cm}
\item[(b)] $l>0$, then $A\in \mathscr{T}(\alpha,\beta)$ if and only if
\begin{equation*}%\label{c5}
t_{s,p}=\left(\frac{\sqrt{|\alpha'(\eta_s)|}\sqrt{|\beta'(\zeta_1)|}}{\sqrt{|\alpha'(\eta_p)|}\sqrt{|\beta'(\zeta_s)|}}\frac{\omega_s^{\alpha}\omega_1^{\beta}}{\omega_p^{\alpha}\omega_s^{\beta}}\frac{\eta_1-\zeta_s}{\eta_p-\zeta_s}t_{1,s}+\frac{\sqrt{|\beta'(\zeta_1)|}}{\sqrt{|\beta'(\zeta_s)|}}\frac{\omega_1^{\beta}}{\omega_s^{\beta}}\frac{\eta_p-\zeta_1}{\eta_p-\zeta_s}t_{1,p}\right)
\end{equation*} for all $p,s$ such that $1\leq p \leq m$, $1\leq  s\leq l$, $s\neq p$, and
\begin{equation*}%\label{c6}
t_{s,p}=\left(\frac{\sqrt{|\alpha'(\eta_1)|}}{\sqrt{|\alpha'(\eta_p)|}}\frac{\omega_1^{\alpha}}{\omega_p^{\alpha}}\frac{\eta_1-\zeta_s}{\eta_p-\zeta_s}t_{s,1}+\frac{\sqrt{|\beta'(\zeta_1)|}}{\sqrt{|\beta'(\zeta_s)|}}\frac{\omega_1^{\beta}}{\omega_s^{\beta}}\frac{\eta_p-\zeta_1}{\eta_p-\zeta_s}t_{1,p}\right)
\end{equation*} for all  $p,s$ such that $1\leq p \leq m$, $l<  s\leq n$.
\end{itemize}
\end{thm}
\begin{proof}
We only give proof for $l=0$. Proof for $l>0$ is analogous.

Let $A$ be any linear transformation from $K_{\alpha}$ into $K_{\beta}$. If $\widetilde{M}_A=(t_{s,p})$ is the matrix representation of $A$ with respect to the bases $\{e_{\eta_1}^{\alpha},\ldots, e_{\eta_m}^{\alpha} \}$ and $\{e_{\zeta_1}^{\beta},\ldots, e_{\zeta_n}^{\beta} \}$, and $M_A=(r_{s,p})$ is the matrix representation of $A$ with respect to the bases $\{v_{\eta_1}^{\alpha},\ldots, v_{\eta_m}^{\alpha} \}$ and $\{v_{\zeta_1}^{\beta},\ldots, v_{\zeta_n}^{\beta} \}$, then
$$t_{s,p}=\langle Ae_{\eta_p}^{\alpha},e_{\zeta_s}^{\beta}\rangle=\omega_p^{\alpha}\overline{\omega}_s^{\beta}\langle Av_{\eta_p}^{\alpha},v_{\zeta_s}^{\beta}\rangle=\frac{\omega_p^{\alpha}}{\omega_s^{\beta}}r_{s,p}.$$

Recall that
$$\omega_{p}^{\alpha}=e^{-\frac{i}{2}(\mathrm{arg}\eta_p-\mathrm{arg}\lambda_1)}.$$
Hence
$$\frac{\eta_p}{\eta_1}=\frac{\lambda_1(\overline{\omega}_p^{\alpha})^2}{\lambda_1(\overline{\omega}_1^{\alpha})^2}=\left(\frac{\omega_1^{\alpha}}{\omega_p^{\alpha}}\right)^2.$$
It is now a matter of a simple computation to show that $\widetilde{M}_A$ satisfies \eqref{c4} if and only if $M_A$ satisfies \eqref{c1}. It then follows from Theorem \ref{macierz3}(a) that $A\in \mathcal{T}(\alpha,\beta)$ if and only if $\widetilde{M}_A$ satisfies \eqref{c4}.
\end{proof}

\end{document}